\documentclass[10pt]{amsart}

\usepackage{amssymb,amsmath,amsthm,amsfonts}

\newcommand{\comment}[1]{}

\newtheorem{lem}{Lemma}
\newtheorem{propn}{Proposition}

\newtheorem{thm}{Theorem}

\theoremstyle{remark}
\newtheorem*{Rem}{Remark}
\theoremstyle{definition}

\newcommand{\R}{\mathbb R}

\newcommand{\Z}{\mathbb Z}
\newcommand{\T}{\mathbb T}

\newcommand{\N}{\mathbb N}
\newcommand{\C}{\mathbb C}

\DeclareMathOperator{\supp}{supp}

\DeclareMathOperator{\lcm}{lcm}

\newcommand{\D}{\delta}
\newcommand{\E}{\epsilon}
\newcommand{\VE}{\varepsilon}
\newcommand{\A}{\alpha}

\newcommand{\lm}{\lambda}

\newcommand{\be}{\begin{equation}}
\newcommand{\ee}{\end{equation}}
\newcommand{\bee}{\begin{equation*}}
\newcommand{\eee}{\end{equation*}}

\begin{document}
\title{An Optimal Version of S\'ark\"ozy's Theorem}
\author{Neil Lyall\quad\quad\quad\'Akos Magyar
}

\address{Department of Mathematics, The University of Georgia, 
Athens, GA 30602, USA}
\email{lyall@math.uga.edu}
\address{Department of Mathematics, University of British Columbia, Vancouver, B.C. V6T 1Z2, Canada}
\email{magyar@math.ubc.ca}

\begin{abstract}
Using Fourier analytic techniques, we prove that if $\VE>0$, $N\geq \exp\exp(C\VE^{-1}\log\VE^{-1})$ and $A\subseteq\{1,\dots,N\}$, then there must exist $t\in\N$ such that
\[\frac{|A\cap (A+t^2)|}{N}>\left(\frac{|A|}{N}\right)^2-\VE.\]

This is a special case of results presented in Lyall and Magyar \cite{LM3} and we will follow those arguments closely. We hope that the exposition of this special case will serve to illuminate the key ideas contained in \cite{LM3}, where many of the analogous arguments are significantly more technical.
\end{abstract}

\maketitle

\setlength{\parskip}{5pt}

\section{Introduction}

A striking and elegant result in density Ramsey theory states that in any subset of the integers of positive upper density there necessarily exist two distinct elements, in fact infinitely many pairs of distinct elements, whose difference is a perfect square. This is equivalent to the following (finite) result:

\begin{thm}\label{1}
Let $A\subseteq[1,N]$ and $\D=|A|/N$. If $N\geq N(\D)$, then there exists $t\ne0$ such that $A\cap(A+t^2)\ne\emptyset$.
\end{thm}

This result was originally conjectured by L. Lov\'asz and eventually verified independently by Furstenberg \cite{Furst} and S\'ark\"ozy \cite{Sarkozy}, using techniques from ergodic theory and Fourier analysis (circle method) respectively.

A simple averaging argument, due to Varnivides (see appendix), shows that Theorem \ref{1} is equivalent to the fact that given any $0<\D\leq1$ there exists a $c(\D)>0$ such that any $A\subseteq[1,N]$ with $|A|=\D N$ must satisfy
\[\frac{1}{N^{1/2}}\sum_{t=1}^{N^{1/2}}\frac{|A\cap(A+t^2)|}{N}\geq c(\D).\]
In particular we can conclude that the set $A$ will contain at least $c(\D)N$ pairs of elements that are the \emph{same} square difference apart.

The purpose of this note is to give an essentially self contained proof
 of the following result, closely following the approach taken in Lyall and Magyar \cite{LM3}. We hope that this exposition will also serve to illuminate the key ideas contained in \cite{LM3}, where many of the analogous arguments are significantly more technical.

\begin{thm}\label{finite1}
Let $A\subseteq[1,N]$ and $\VE>0$. If $N\geq\exp\exp(C\VE^{-1}\log\VE^{-1})$, then there exists $t\ne0$ such that
\be\label{optimal intersect}
\dfrac{|A\cap(A+t^2)|}{N}>\left(\dfrac{|A|}{N}\right)^2-\VE.
\ee
\end{thm}

We note that in general the lower bound in (\ref{optimal intersect}) is sharp, or rather ``$\VE$-optimal". This can be seen by considering, for example, random sets.


\section{Preliminaries}

\subsection{Fourier analysis on $\Z$}
If $f:\Z\rightarrow\C$ is a function for which $\sum_{n\in\Z}|f(n)|<\infty$ we will say that $f\in L^1=L^1(\Z)$ and define 
\be
\|f\|_1=\sum_{n\in\Z}|f(n)|.\ee
For $f\in L^1$ we define its \emph{Fourier transform} $\widehat{f}:\T\rightarrow\C$ by
\be
\widehat{f}(\A)= \sum\limits_{n\in\Z}f(n)e^{-2\pi i n\A}.\ee

The summability assumption on $f$ ensures that $\widehat{f}$ is a continuous function on the circle $\T$ (which we will identify with the interval of real numbers $[0,1]$) and that in this setting the Fourier inversion formula and Plancherel's identity, namely
\[f(n)=\int_{0}^1\widehat{f}(\A)e^{2\pi i n\A}d\A
\quad\quad\text{and}\quad\quad
\int_0^1|\widehat{f}(\A)|^2 d\A=\sum_{n\in\Z}|f(n)|^2\]
are simply immediate consequences of the familiar orthogonality relation
\[\int_0^1e^{2\pi i n \A}d\A=\begin{cases}
1\quad\text{if \ $n=0$}\\ 
0\quad\text{if \ $n\ne0$}
\end{cases}.\]

Defining the convolution of $f$ and $g$ to be
\[f*g(n)=\sum_{\ell\in\Z}f(n-\ell)g(\ell)\]
it follows that if $f,g\in L^1$ then $f*g\in L^1$ with
\[\|f*g\|_1\leq\|f\|_1\|g\|_1\quad\quad\text{and}\quad\quad\widehat{f*g}=\widehat{f}\,\widehat{g}.\]

Finally we remark that it follows from the Poisson Summation Formula that if $\psi\in\mathcal{S}(\R)$, then
\be
\widehat{\psi}(\A)=\sum_{n\in\Z}\widetilde{\psi}(\A-n)
\ee
where
\[\widetilde{\psi}(\xi)=\int_{\R}\psi(x)e^{-2\pi i x\xi}\,dx\]
denotes the Fourier transform (on $\R$) of $\psi$.

\subsection{Counting square differences 
}

Let $A\subseteq[1,N]$ and $\D=|A|/N$. 

Let $1\leq\mu\leq\lm$ be integers with $\lm^2\leq N/4$.
It is easy to verify, using the properties of the Fourier transform discussed above, that the average number of pairs of elements in $A$ whose difference is equal to the square of an integer 
$t\in(\lm,\lm+\mu]$ can be expressed as follows:
\bee\frac{1}{\mu}\sum_{t=\lm+1}^{\lm+\mu}|A\cap(A+t^2)|=\frac{1}{\mu}\sum_{t=\lm+1}^{\lm+\mu}\sum_{n\in\Z}1_A(n)1_A(n-t^2)=\int_0^1|\widehat{1_A}(\A)|^2 S_{\lm,\mu}(\A)\,d\A\eee
where 
\be\label{weyl sum}
S_{\lm,\mu}(\A)=\frac{1}{\mu}\sum_{t=\lm+1}^{\lm+\mu} e^{2\pi i t^2\A}=\frac{1}{\mu}\sum_{t=1}^{\mu} e^{2\pi i (t^2+2\lm t+\lm^2)\A }
\ee
is a classical (normalized) Weyl sum. 

\subsection{Standard Weyl sum estimates}

It is clear that whenever $|\A|\ll \mu^{-2}$ there can be no cancellation in the quadratic Weyl sum (\ref{weyl sum}), in fact the same is also true when $\A$ is close to a rational with \emph{small} denominator (i.e. there is no cancellation over sums in residue classes modulo $q$). 

We now state a precise formulation of the well known fact that this is indeed the only obstruction to cancellation. 
Lemma \ref{Weyl Estimates} is usually stated with weaker hypotheses, namely with $q$ in place of $q^2$ in (\ref{q^2}).
For a proof of this stronger result see \cite{LM1} or \cite{LM1'}.

\begin{lem}\label{Weyl Estimates}
Let $\eta>0$. If  
\be\label{q^2}
\left|\A-\frac{a}{q^2}\right|>\frac{1}{\eta^2\mu^2}
\ee 
for all $a\in\Z$ and $1\leq q\leq \eta^{-2}$,
then
\be\label{minorest}
\left|S_{\lm,\mu}(\A)\right|\leq C_1\eta.
\ee
\end{lem}

\begin{Rem}
It is easy to see that one can conclude from Lemma \ref{Weyl Estimates} that estimate (\ref{minorest}) also holds (under the same hypotheses as above with say  $C_1$ replaced with $2C_1$) for the ``perturbed'' Weyl sums
\[\frac{1}{\mu}\sum_{t\in(\lm,\lm+\mu]\cap\Z} e^{2\pi i t^2\A }\]
where $1\leq\mu\leq\lm$ are now no longer assumed to take on integer values, provided $\mu\gg\eta^{-1}$.
\end{Rem}

Note that Lemma \ref{Weyl Estimates}, together with the Plancherel identity, allows us to conclude that 
\[\int_0^1|\widehat{1_A}(\A)|^2 S_{\lm,\mu}(\A)\,d\A=\int_{\mathfrak{M}_{\eta,\mu}} |\widehat{1_A}(\A)|^2 S_{\lm,\mu}(\A)\,d\A+O(\eta N)\]
where \[\mathfrak{M}_{\eta,\mu}=\bigcup_{q=1}^{\eta^{-2}}\bigcup_{a=0}^{q^2-1}\left\{\A\in[0,1]\,:\, \left|\A-\frac{a}{q^2}\right|\leq \frac{1}{\eta^2\mu^2} 
\right\}.\]

However, in order to carry out our Fourier analytic arguments it will be convenient  to consider the set of equally spaced rational numbers in $[0,1]$ with denominator 
\be
q_\eta=\lcm\{1\leq q \leq \eta^{-2}\}
\ee
as opposed to the much smaller, but alas more wildly distributed, set of rational numbers described above.  

Note that it follow from elementary considerations involving prime numbers that $q_\eta\leq \exp(C\eta^{-2})$ and this accounts for one of the exponentials in the bound in Theorem \ref{1}.


\section{The Dichotomy Proposition}

We now state our key dichotomy proposition (that is stronger than we actually need for the purposes of this note) 
and demonstrate how it can be used to prove Theorem \ref{finite1} (we could have simplify matters and taken $\mu=\lm$ everywhere below).
The arguments in this section are close in spirit to, and very much influenced by, those of Bourgain \cite{Bourgain}, see also Magyar \cite{Magyar}.

Let $\eta>0$ and $1\leq\mu\leq\lm$. We define
\be
\Omega_{\eta,\lm,\mu}=\left\{\A\in[0,1]\,:\, \frac{\eta^2}{\lm^2}\leq\Bigl|\A-\frac{a}{q_\eta^2}\Bigr|\leq\frac{1}{\eta^2\mu^2}\ \text{for some $a\in\mathbb{Z}$}\right\}
\ee
where
$q_\eta=\lcm\{1\leq q \leq \eta^{-2}\}.$

\begin{propn}\label{dichotomy} Let $A\subseteq[1,N]$, $\D=|A|/N$, and $0<\VE\leq\D^2$. 
Let $\eta_\VE=\exp(-C\VE^{-1}\log\VE^{-1})$ and $q_\VE=q_{\eta_\VE}$.

If $1\leq\mu\leq\lm$ are any given pair of integers that satisfy
$\mu\gg\eta_\VE^{-1}q_\VE$ and $N\gg\eta_\VE^{-2}\lm^2$
then either
\be\label{random}
|A\cap(A+t^2)|>(\D^2-\VE)N \ \  \text{for some \ $t\in (\lm,\lm+\mu]\cap\Z$}
\ee
or
\be\label{roughint}
\int_{\Omega} |\widehat{1_A}(\A)|^2\,d\A\geq\VE N/10 
\ee
where $\Omega=\Omega_{\eta_\VE,\lm,\mu}$.
\end{propn}

Proposition \ref{dichotomy} expresses, in our setting, the basic dichotomy that either $A$ behaves as though it were a random set, or has arithmetic structure as the Fourier transform $\widehat{1_A}$ is concentrated (on small annuli) around a fixed number of equally spaced rational points.

We shall see below that one can in fact replace (\ref{random}) in Proposition \ref{dichotomy} with the stronger statement that a positive proportion of the integers $t\in(\lm,\lm+\mu]$ satisfy the estimate $|A\cap(A+t^2)|>(\D^2-\VE)N$. More precisely (\ref{random}) can be replaced by 
\be
\left|\left\{t\in(\lm,\lm+\mu]\cap\Z\,:\,|A\cap(A+t^2)|>(\D^2-\VE)N\right\}\right|\geq \frac{c\VE}{q_{\VE/2}}\mu.
\ee

\subsection{Proposition \ref{dichotomy} implies Theorem \ref{finite1}}
Let $\VE>0$,  $\eta_\VE=\exp(-C\VE^{-1}\log \VE^{-1})$ and $q_\VE=q_{\eta_\VE}$. Fix an integer $J>10/\VE$ and let $\{\lm_j\}_{j=1}^J$ be any sequence of integers with the property that $\lm_1= C\eta_\VE^{-1} q_\VE$ and
\[\lm_j\leq \eta_\VE^{2} \lm_{j+1}\]
for $1\leq j<J$. 
It is easy to now see that the sets
$\Omega_j=\Omega_{\eta_\VE,\lm_j,\lm_j}$
are disjoint.

Suppose there exists $N\geq C\eta_\VE^{-2}\lm_J^2$ and a set $A\subseteq[1,N]$ such that
\be\label{negative}
\dfrac{|A\cap(A+t^2)|}{N}\leq \left(\dfrac{|A|}{N}\right)^2-\VE
\ee
for all integers $1\leq t\leq \sqrt{N}$. An application Proposition \ref{dichotomy} allows us to conclude that for such a set one must have
\be
\sum_{j=1}^J\int_{\Omega_j} |\widehat{1_A}(\A)|^2\,d\A\geq J\VE N/10>N
\ee
since in particular (\ref{negative}) must hold for all integers $t\in\bigcup_{j=1}^J(\lm_j,2\lm_j]$.

On the other hand it follows from the disjointness property of the sets $\Omega_j$ (which we guarantee by our initial choice of sequence $\{\lm_j\}$) and Plancherel's Theorem that
\be
\sum_{j=1}^J\int_{\Omega_j} |\widehat{1_A}(\A)|^2\,d\A\leq \int_0^1|\widehat{1_A}(\A)|^2\,d\A\leq |A|\leq N
\ee
giving a contradiction. The result now follows since we can clearly choose $\lm_J$ such that \[\eta_\VE^{-2}\lm_J^2=\exp(C\eta_\VE^{-2})=\exp\exp(C\VE^{-1}\log\VE^{-1}).\]


\section{Establishing a smooth variant of Proposition \ref{dichotomy}}

We now formulate a functional variant of Proposition \ref{dichotomy} that is well suited to our Fourier analytic approach. 

\subsection{Counting function}


For $g,h:[1,N]\rightarrow[0,1]$ and $q,\lm,\mu\in\N$ we define
\be
\Lambda_q(g,h)=\frac{q}{\mu}\sum_{\substack{t\in(\lm,\lm+\mu] \\ q|t}}\sum_{n\in\Z}g(n)h(n-t^2).
\ee

With $g=h=1_A$ this essentially gives a normalized count for the number of pairs of elements in $A$ whose difference is equal to the square of an integer
$t\in(\lm,\lm+\mu]$ for which $q|t$.

Note that it is natural to consider only those $t\in\N$ that are divisible by some (large) natural number $q$. Indeed, as a consequence of the fact that our set $A$ could fall entirely into a single congruence class $\pmod{d}$, with $1\leq d\le\VE^{-1/2}$, it follows that if there were to exist $t\in\N$ such that $A\cap(A+t^2)\ne\emptyset$ for an arbitrary set $B$, then these $t$ would necessarily have to be divisible by all $1\leq d\leq \VE^{-1/2}$ and hence by the least common multiple of all $1\leq d\leq \VE^{-1/2}$, a quantity of size $\exp(C\VE^{-1/2})$.

As before this can be expressed on the transform side as
\be\label{FTside}
\Lambda_{q}(g,h)=\int_{\T^k} \widehat{g}(\A)\overline{\widehat{h}(\A)}S_{\lm,\mu,q}(\A)\,d\A
\ee

where now
\be
S_{\lm,\mu,q}(\A)=\frac{q}{\mu}\sum_{\substack{t\in(\lm,\lm+\mu] \\ q|t}}e^{2\pi i\,t^2\A}
\ee

\begin{Rem}
If the integers $\lm$ and $\mu$ are both divisible by $q$, then it is easy to relate $S_{\lm,\mu,q}$ to the ``classical'' Weyl sum discussed above. In fact, one can easily verify that
\be
S_{\lm,\mu,q}(\A)=S_{\lm/q,\mu/q}(q^2\A).
\ee
\end{Rem}

\subsection{A smooth variant of Proposition \ref{dichotomy}}

Let $\psi:\R\rightarrow(0,\infty)$ be a Schwartz function satisfying 
\[\widetilde{\psi}(0)=1\geq\widetilde{\psi}(\xi)\geq0\quad\quad\text{and}\quad\quad \widetilde{\psi}(\xi)=0 \ \ \text{for} \ \ |\xi|>1.\]

For a given $q\in\N$ and $L>1$ we define
\be\label{normalizedpsi}
\psi_{q,L}(x)=\begin{cases}
\left(\frac{q}{L}\right)^2\psi\left(\frac{q^2\ell}{L^2}\right)& \ \ \text{if} \ \ x=q^2\ell \ \ \text{for some $\ell\in\Z$}\\
 \ 0& \ \ \text{otherwise}
\end{cases}
\ee

It follows from the Poisson summation formula that the Fourier transform (on $\Z$) of $\psi_{q,L}$ takes the form
\be
\widehat{\psi}_{q,L}(\A)
=\sum_{\ell\in\Z}\widetilde{\psi}\left(L^2\left(\A-\ell/q^2\right)\right).
\ee

Note that $\widehat{\psi}_{q,L}$ is supported on
\[M_{q,L}= \left\{\A\in[0,1]\,:\,\Bigl|\A-\frac{a}{q^2}\Bigr|\leq\frac{1}{L^2}\ \text{for some $a\in\mathbb{Z}$}\right\}\]
and that if our cutoff function $\psi$ is chosen appropriately, then
\be
\widehat{\psi}_{q_\VE,\eta_\VE \mu}-\widehat{\psi}_{q_\VE,\VE\eta_\VE^{-1}\lm}
\ee
will be essentially supported on
\[\Omega_{\eta_\VE,\lm,\mu}=M_{q_\VE,\eta_\VE \mu}\setminus M_{q_\VE,\eta_\VE^{-1}\lm}\]
in the sense that 
\be\label{cutoff}
\bigl| \widehat{\psi}_{q_\VE,\eta_\VE \mu}(\A)-\widehat{\psi}_{q_\VE,\VE\eta_\VE^{-1}\lm}(\A)\bigr|\leq\VE/10
\ee
whenever $\A\notin\Omega_{\eta_\VE,\lm,\mu}$.

\begin{propn}[Smooth functional variant of Proposition \ref{dichotomy}]\label{smoothdichotomy}

Let $f:[1,N]\rightarrow[0,1]$ and $\D=N^{-1}\sum_{x\in\Z}f(x)$. 

Let $0<\VE\leq\D^2$ and $1\leq\mu\leq\lm$ be any given pair of integers that satisfy
$\mu\gg\eta_\VE^{-1}q_\VE$ and $N\gg\eta_\VE^{-2}\lm^2$ where
$q_\VE=q_{\eta_\VE}$ with $\eta_\VE=\exp(-C\VE^{-1}\log\VE^{-1})$. 
Then there exists $0<\eta\ll\VE$ satisfying $\eta_\VE\leq\VE\eta$, 
such that either
\be\label{SR}
\Lambda_{q}(f,f)>(\D^2-\VE)N
\ee
or
\be\label{smoothint}
\int_0^1|\widehat{f}(\A)|^2\bigl|\widehat{\psi}_{q,L_2}(\A)-\widehat{\psi}_{q,L_1}(\A)\bigr|\,d\A\geq\VE N/5
\ee
where $L_1=\eta^{-1}\lm$, $L_2=\eta \mu$, and $q=q_\eta$.  
\end{propn}

\subsection{Proposition \ref{smoothdichotomy} implies Proposition \ref{dichotomy}}

Let $f=1_A$ and $q=q_\eta$,
noting that $q\leq q_\VE$. 

It is easy to see that if $\Lambda_{q}(f)>(\D^2-\VE)N$, then
\be
\left|\left\{t\in(\lm,\lm+\mu]\cap\Z\,:\,|A\cap(A+t^2)|>(\D^2-2\VE)N\right\}\right|\geq \frac{c\VE}{q}\mu\geq \frac{c\VE}{q_\VE}\mu
\ee
which is precisely the strengthening of (\ref{random}) that we eluded to above (with $2\VE$ in place of $\VE$).

While from the fact that $q|q_\VE$ it follows that
\[\supp\bigl(\widehat{\psi}_{q,L_2}-\widehat{\psi}_{q,L_1}\bigr)\subseteq \supp\bigl(\widehat{\psi}_{q_\VE,\eta_\VE \mu}-\widehat{\psi}_{q_\VE,\VE\eta_\VE^{-1}\lm}\bigr)\]
and hence from the remarks preceding Proposition \ref{smoothdichotomy} (in particular (\ref{cutoff})) that (\ref{smoothint}) implies (\ref{roughint}).


\section{Proof of Proposition \ref{smoothdichotomy}}

\subsection{Decomposition}
Let
$f:[1,N]\rightarrow[0,1]$ and $\D=N^{-1}\sum_{n\in\Z}f(n)$.

We make the decomposition
\be
f=f_1+f_2+f_3
\ee
where
\be
f_1=f*\psi_{q,L_1}\quad\text{and}\quad
f_2=f-f*\psi_{q,L_2}
\ee
which of course forces 
\be
f_3=f*(\psi_{q,L_2}-\psi_{q,L_1}).
\ee

One should think of $f_1(n)$ (respectively $f*\psi_{q,L_2}(n)$) as being essentially the average value of the function $f$ over arithmetic progressions of difference $q^2_\eta$ and (total) length $L^2_1=\eta^{-2}\lm^2$ (respectively $L^2_2=\eta^2 \mu^2$) centered at $n$.

\subsection{Proof of Proposition \ref{smoothdichotomy}}

Note that
\be
\Lambda_{q}(f,f)=\Lambda_{q}(f_1,f_1)+\underbrace{\Lambda_{q}(f_2,f_1)+\Lambda_{q}(f,f_2)}_{(\star)}+\underbrace{\Lambda_{q}(f_3,f_1)+\Lambda_{q}(f,f_3)}_{(\star\star)}
\ee
where both terms in $(\star)$ involve a $f_2$ and both terms in $(\star\star)$ involve a $f_3$.

The proof of Proposition \ref{smoothdichotomy} will follow as an almost immediate consequence of the following two lemmas.

\begin{lem}[Main term]\label{lemma1}
Let $\VE>0$. If $0<\eta\ll \VE$, 
then
\be
\Lambda_q(f_1,f_1)\geq (\D^2-\VE/2)N
\ee
\end{lem}

\begin{lem}[Error term]\label{lemma3}
Let $\VE>0$, then there exists $\eta>0$ satisfying $\exp(-C'\VE^{-1}\log\VE^{-1})\leq\eta\ll\VE$, such that
\be\label{wholepoint}
\|(1-\widehat{\psi}_{q,L_2})S_{\lm,\mu,q}\|_\infty\leq\VE/20
\ee 
and hence
\be
|\Lambda_{q}(f_2,f_1)+\Lambda_{q}(f,f_2)|\leq (\VE/10)N.
\ee
\end{lem}

\begin{proof}[Proof of Proposition \ref{smoothdichotomy}]
If
$\Lambda_q(f,f)\leq(\D^2-\VE)N$, then it follows from
Lemma \ref{lemma1} that
\[|\Lambda_q(f,f)-\Lambda_q(f_1,f_1)|\geq(\VE/2)N.\]
Since
\[|\Lambda_{q}(f_3,f_1)+\Lambda_{q}(f,f_3)|\geq |\Lambda_{q}(f,f)-\Lambda_{q}(f_1,f_1)|-
|\Lambda_{q}(f_2,f_1)+\Lambda_{q}(f,f_2)|\]
it consequently follows from Lemma \ref{lemma3} that
\[|\Lambda_{q}(f_3,f_1)+\Lambda_{q}(f,f_3)|\geq (2\VE/5)N.\]

The proposition then follows from the observation that
\be\label{max}
\max\{|\Lambda_{q}(f_3,f_1)|,|\Lambda_{q}(f,f_3)|\}
\leq \int_{0}^1|\widehat{f}(\A)|^2\bigl|\widehat{\psi}_{q,L_2}(\A)-\widehat{\psi}_{q,L_1}(\A)\bigr|\,d\A.
\ee
which follows from standard properties of convolutions under the action of the Fourier transform, identity (\ref{FTside}), and trivial bounds for the exponential sum $S_{\lm,\mu,q}$.\end{proof}


\subsection{Proof of Lemma \ref{lemma1}} Let $q=q_\eta$ and recall that $L_1=\eta^{-1}\lm$.
If $q|t$ and $\lm<t\leq\lm+\mu\leq2\eta L_1$, then it is straightforward to see that $\psi$ can be chosen such that $f_1$ is essentially invariant under translation by $t^2$ in the the sense that
\bee
\left|f_1(n)-f_1(n-t^2)\right|=\frac{q^2}{L_1^2}\sum_{\ell\in\Z}\left|\psi\left(\frac{q^2\ell-t^2}{L_1^2}\right)-\psi\left(\frac{q^2\ell}{L_1^2}\right)\right|\leq c\eta^2
\eee
for some constant $c>0$. 
Therefore, provided $\eta$ is chosen so that $c\eta^2\leq \VE/4$, we have
\bee
\Lambda_q(f_1)
\geq \sum_{n\in\Z} f_1(n)^2 -\frac{\VE}{4}\sum_{n\in\Z}f_1(n).
\eee

Since $\psi_{q,L_1}$ is $L^1$-normalized it follows that
\bee
\sum_{n\in\Z}f_1(n)=\sum_{n,m\in\Z}f(n-m)\psi_{q,L_1}(m)=\sum_{n\in\Z}f(n)=\D N.
\eee
Using Cauchy-Schwarz, one obtains
\bee
\sum_{n\in\Z} f_1(n)^2\geq\sum_{-\VE N/16\leq n\leq N+\VE N/16} f_1(n)^2
\geq  \frac{1}{(1+\VE/8) N} \left(\sum_{-\VE N/16\leq n\leq N+\VE N/16} f_1(n)\right)^2.
\eee
Since $f$ is supported on $[1,N]$ (and $\psi_{q,L_1}$ is $L^1$-normalized) it follows that
\bee
\sum_{-\VE N/16\leq n\leq N+\VE N/16} f_1(n)\geq \sum_{n\in\Z}f(n)\left(1-\sum_{|m|\geq\VE N/16} \psi_{q,L_1}(m)\right)\geq \D N (1-\VE/16)
\eee
as $\psi$ can be chosen so that $\sum_{|m|\geq\E N} \psi_{q,L_1}(m)\leq \E$ whenever $N\gg L_1$.
\qed

\subsection{Proof of Lemma \ref{lemma3}}

It is in establishing Lemma \ref{lemma3} that we finally exploit the arithmetic properties of the set of squares. In particular, we will make use of the following ``minor arc estimates'' for the exponential sums $S_{\lm,\mu,q}$.

\begin{lem}[Corollary of Lemma \ref{Weyl Estimates}]\label{Weyl Estimates 2}
Let $\VE>0$. If $0<\eta\ll\VE$ and $0<\eta'<\VE\eta$, then
\be\label{MainSquares}
\|(1-\widehat{\psi}_{q',L'_2})S_{\lm,\mu,q}\|_\infty\leq 2C_1\eta'/\eta
\ee
where  $q'=q_{\eta'}$ and $L'_2=\eta'\mu$.
\end{lem}
\begin{proof}
Let $\eta_0=\eta'/\eta$ and  $\A\in[0,1]$ be fixed. 
If there exists $a\in\Z$ such that
\[\left|\A-\frac{a}{q'^2}\right|\leq\frac{\VE}{(\eta'\mu)^2},\]
then (as remarked earlier) $\psi$ can be chosen such that
\be\label{*}
|1-\widehat{\psi}_{q',L'_2}(\A)|\leq \VE.\ee

While if
\bee
\left|\A-\frac{a}{q'^2}\right|>\frac{\VE}{(\eta'\mu)^2}
\eee
for all $a\in\Z$,
then
\[\left|q^2\A-\frac{a}{q_0^2}\right|>\frac{q^2}{(\eta_0)^2\mu^2}\]
for all $a\in\Z$, since $q q_0|q'$ where $q_0=q_{\eta_0}$.

It therefore follows from the fact that
\[S_{\lm,\mu,q}(\A)=\frac{1}{\mu'}\sum_{s\in(\lm'+q^{-1},\lm'+\mu']\cap\Z} e^{2\pi i s^2(q^2\A)}\]
where $\lm'=\lm/q$ and $\mu'=\mu/q$ and the remark proceeding Lemma \ref{Weyl Estimates} 
that
\be\label{minorest 2}
\left|S_{\lm,\mu,q}(\A)\right|\leq 2C_1\eta_0.
\ee
Estimate (\ref{MainSquares}) follows immediately from (\ref{*}) and (\ref{minorest 2}).
\end{proof}

\begin{proof}[Proof of Lemma \ref{lemma3}]

We first construct the number $\eta>0$. Choosing a lacunary sequence $\{\eta_j\}$ for which 
\bee
\eta_1\ll \VE\quad \text{and}\quad \eta_{j+1} \leq(\VE/40C_1)\eta_{j}
\eee 
for each $j\geq1$ it is easy to see that
\[\sup_{\A\in[0,1]}\sum_{j=1}^\infty\bigl|\widehat{\psi}_{j+1}(\A)-\widehat{\psi}_j(\A)\bigr|\leq C_2\]
where $\widehat{\psi}_j=\widehat{\psi}_{q_{j},\eta_j\mu}$ with $q_j=q_{\eta_j}$. It follows immediately that there must exist $1\leq j\leq 40C_2/\VE$ such that
\be\label{80}\|\widehat{\psi}_{j+1}-\widehat{\psi}_j\|_\infty\leq  \VE/40.\ee
We set $\eta=\eta_j$ and $\eta'=\eta_{j+1}$ for this value of $j$ and note that $\eta$ satisfies the inequality 
\[\exp(-C'\VE^{-1}\log\VE^{-1})\leq\eta\ll\VE.\]

Estimate (\ref{wholepoint}) now follows immediately from Lemma \ref{Weyl Estimates 2} and (\ref{80}), since 
\be
\|(1-\widehat{\psi}_{q,L_2})S_{\lm,\mu,q}\|_\infty\leq\|(1-\widehat{\psi}_{q',L_2})S_{\lm,\mu,q}\|_\infty+\|(\widehat{\psi}_{q,L_2}-\widehat{\psi}_{q',L_2})S_{\lm,\mu,q}\|_\infty\leq 2C_1\eta'/\eta+\VE/40
\ee
and $\eta'/\eta\leq \VE/80C_1$.

Lemma \ref{lemma3} now follows, since by
arguing as in the proof of Proposition \ref{smoothdichotomy} above, we obtain
\begin{align*}
\max\{|\Lambda_{q}(f_2,f_1)|,|\Lambda_{q}(f,f_2)|\}&\leq \int_{0}^1 |\widehat{f}(\A)|^2\,\bigl|1-\widehat{\psi}_{q,L_2}(\A)\bigr||S_{\lm,\mu,q}(\A)|\,d\A\\
& \leq \|(1-\widehat{\psi}_{q,L_2})S_{\lm,\mu,q}\|_\infty N\end{align*}
where the last inequality follows from Plancherel and the fact that $\|f\|_2^2\leq\|f\|_1\leq N$. 
\end{proof}


\begin{appendix}
\section{A Varnavides-type theorem for square differences}

The purpose of this section is to prove the following theorem. 

\begin{thm}\label{3}
Let $0<\D\leq1$. There exists $c=c(\D)$ such that if $A\subseteq[1,N]$ with $|A|=\D N$, then
\[\sum_{t=1}^{N^{1/2}}|A\cap(A+t^2)|\geq cN^{3/2}.\]
\end{thm}

Theorem \ref{3} strengthens S\'ark\"ozy's theorem (Theorem \ref{1}) in the same way in which a theorem of Varnavides \cite{V} strengthens Roth's theorem on arithmetic progressions of length three. It guarantees the existence of ``many'' square difference in a set of positive density, instead of just one.

The proof of this result combines S\'ark\"ozy's theorem with a modification of Varnavides' original combinatorial argument \cite{V}.
We will closely follow the presentation given in \cite{HL} and \cite{Sound}.

\begin{proof}
Let $A\subseteq[1,N]$ such that $|A|=\D N$ with $N$ sufficiently large. 
By S\'ark\H{o}zy's theorem we know that there exists $M=M(\D)$ such that any set with at least $\D M/2$ elements in $[1,M]$ will contain a non-trivial square difference.

Now consider the arithmetic progressions
\[P_{n,t}=\{n,n+t^2,\dots,n+(M-1)t^2\}\subseteq[1,N]\]
with $t^2\leq \D N/M^2$ and $n\leq N(1-\D/M)$. 

We say that such a progression $P_{n,t}$ is \emph{good} if
\[\frac{|A\cap P_{n,t}|}{M}\geq\frac{\D}{2}.\]
A simple counting argument shows that there are at least
$(\D N)^{3/2}/M$ \emph{good} progressions $P_{n,t}$.

By S\'ark\H{o}zy's theorem each good progression contributes at least one square difference in $A$. But of course some of these square differences could get over counted. Suppose we are 
given a pair $\{n,n+s^2\}$ in $A$. If this pair is contained in $P_{n,t}$, then $t$ must be a divisor of $s$ and moreover $s^2\leq Mt^2$. It therefore follows that there are at most $M$ choices for $t$ and it is easy to see that each choice of $t$ fixes $n$ in at most $M$ ways. Therefore each square difference is over 
counted at most $M^2$ times. 

It follows that $A$ must contain at least $(\D N)^{3/2}/M^3=c(\D)N^{3/2}$
distinct square differences, as required
\end{proof}
\end{appendix}

\end{document}